\documentclass[12pt]{amsart}
\usepackage{amssymb,amsmath,amsfonts,amscd}
\usepackage{times,graphicx,epsfig}
\usepackage{soul}

\theoremstyle{plain}
\newtheorem{theorem}{Theorem}[section]

\newtheorem{proposition}[theorem]{Proposition}
\newtheorem{definition}[theorem]{Definition}

\newtheorem{corollary}[theorem]{Corollary}
\newtheorem{remark}[theorem]{Remark}
\newtheorem{lemma}[theorem]{Lemma}

\newcommand{\gf}{\varphi}

\newcommand{\gs}{\sigma}
\newcommand{\gt}{\theta}

\usepackage{pdfsync}

\usepackage{flexisym}

\newcommand{\rn}[1]{{\mathbb R}^{#1}}
\newcommand{\R}{\mathbb R}

\newcommand{\G}{\mathbb G}

\newcommand{\supp}{\mathrm{supp}\;}

\newcommand{\cov}[1]{{\bigwedge\nolimits^{#1}{\mfrak g}}}
\newcommand{\vet}[1]{{\bigwedge\nolimits_{#1}{\mfrak g}}}

\newcommand{\covw}[2]{{\bigwedge\nolimits^{#1,#2}{\mfrak g}}}

\newcommand{\scal}[2]{\langle {#1} , {#2}\rangle}

\newcommand{\Scal}[2]{\langle {#1} \vert {#2}\rangle}
\newcommand{\scalp}[3]{\langle {#1} , {#2}\rangle_{#3}}

\newcommand{\ccheck}{{\vphantom i}^{\mathrm v}\!\,}
\newcommand{\mc}{\mathcal }

\newcommand{\mfrak}{\mathfrak}

\usepackage[usenames]{color}

\begin{document}


\bigskip

\title[Primitives of volume forms in Carnot groups]{Primitives of volume forms in Carnot groups}

\author[Annalisa Baldi, Bruno Franchi, Pierre Pansu]{
Annalisa Baldi\\
Bruno Franchi\\ Pierre Pansu
}

\begin{abstract} In the Euclidean space  it is known
that a function $f\in L^2$ of a ball, with vanishing average,
is the divergence of a vector field $F\in L^2$ with
$$
\| F\|_{ L^2(B)} \le C \|f\|_{L^2(B)}.
$$
In this Note we prove a similar result in any Carnot group $\G$
for a vanishing average $f\in L^p$, $1\le p < Q$, where $Q$ is the so-called homogeneous
dimension of $\G$.

\end{abstract}
 


\maketitle

\tableofcontents

\section{Introduction}

This note is motivated by a question raised by Michael Cowling \cite{C}: in \(\mathbb{R}^n\), it is known that a function \( f \in L^2(B_{\mathrm{Euc}}(0,1)) \) with vanishing average can be expressed as the divergence of a vector field \( F \in L^2(B_{\mathrm{Euc}}(0,1))^n \), satisfying
\[
\|F\|_{L^2(B_{\mathrm{Euc}}(0,1))^n} \le C \|f\|_{L^2(B_{\mathrm{Euc}}(0,1))}.
\]
The question is whether a similar result holds in Heisenberg groups, which can be identified with \(\mathbb{R}^{2n+1}\).

This problem can be rephrased in terms of Sobolev inequalities for differential forms in the Rumin complex \((E_0^\bullet, d_c)\) (see Section \ref{rumin complex} for precise definitions). Specifically, given a compactly supported volume form \(\omega = f \, dV \) with vanishing average, does there exist a $(n-1)$-compactly supported primitive \(\phi\) whose \(L^2\)-norm is controlled by the \(L^2\)-norm of $\omega$?

Sobolev inequalities for the Rumin complex in Heisenberg groups have been studied in \cite{BFP2}, but unfortunately, the results in \cite{BFP2} do not cover the case of volume forms. The aim of this paper is to fill this gap by providing a positive answer to Cowling's question. Furthermore, the results of this note are formulated in the more general setting of Carnot groups (of which Heisenberg groups are a special case), and the \(L^2\)-norms are replaced with any suitable \(L^p\)-norms.

The main result is presented in Theorem \ref{main result} in Section \ref{main result and proof} (see also Theorem \ref{main result bis} for an equivalent formulation). Section \ref{preliminar} provides some preliminary definitions, while
Section \ref{rumin complex} gives a brief introduction to Rumin's complex (for more details, see \cite{rumin_grenoble}, \cite{FS2} and \cite{BFTT}).
Finally, Section \ref{carnot kernels} is an appendix which collects various results on convolution kernels in Carnot groups, some of which are well-known.

\section{Preliminary Results and Notations} \label{preliminar}

A \textit{Carnot group} \(\G\) of step \(\kappa\) is a connected, simply connected Lie group whose Lie algebra \(\mathfrak{g}\) has dimension \(n\) and admits a \textit{step \(\kappa\) stratification}. This means there exist linear subspaces \(V_1, \dots, V_\kappa\) such that
\begin{equation}\label{stratificazione}
\mathfrak{g} = V_1 \oplus \dots \oplus V_\kappa, \quad [V_1, V_i] = V_{i+1}, \quad V_\kappa \neq \{0\}, \quad V_i = \{0\} \, \text{for} \, i > \kappa,
\end{equation}
where \([V_1, V_i]\) is the subspace of \(\mathfrak{g}\) generated by the commutators \([X, Y]\) with \(X \in V_1\) and \(Y \in V_i\). Let \(m_i = \dim(V_i)\) for \(i = 1, \dots, \kappa\), and define \(h_i = m_1 + \dots + m_i\), with \(h_0 = 0\) and, clearly, \(h_\kappa = n\). 

Choose a basis \(\{e_1, \dots, e_n\}\) of \(\mathfrak{g}\), adapted to the stratification, such that
\[
e_{h_{j-1}+1}, \dots, e_{h_j} \, \text{is a basis of} \, V_j \, \text{for each} \, j = 1, \dots, \kappa.
\]
This basis \(\{e_1, \dots, e_n\}\) will be fixed throughout this note.

Let \(X = \{X_1, \dots, X_n\}\) be the family of left-invariant vector fields such that \(X_i(0) = e_i\). Given \eqref{stratificazione}, the subset \(X_1, \dots, X_{m_1}\) generates, by commutations, all the other vector fields. We will refer to \(X_1, \dots, X_{m_1}\) as the \textit{generating vector fields} of the group. 

The Lie algebra \(\mathfrak{g}\) can be endowed with a scalar product \(\langle \cdot, \cdot \rangle\), making \(\{X_1, \dots, X_n\}\) an orthonormal basis. The group \(\G\) can be identified with its Lie algebra \(\mathfrak{g}\), endowed with the product defined by the Hausdorff-Campbell-Dynkin formula. Therefore, on \(\G\), a scalar product, still denoted by \(\langle \cdot, \cdot \rangle\), is well defined.

A point \(p \in \G\) can be written as \(p = (p_1, \dots, p_n)\) or as \(p = p^{(1)} + \dots + p^{(\kappa)}\), where \(p^{(i)} \in V_i\) for \(i = 1, \dots, \kappa\).

Two important families of automorphisms of \(\G\) are the group translations and dilations. For any \(x \in \G\), the \textit{(left) translation} \(\tau_x : \G \to \G\) is defined as
\[
z \mapsto \tau_x z := x \cdot z.
\]
For any \(\lambda > 0\), the \textit{dilation} \(\delta_\lambda : \G \to \G\) is defined as
\[
\delta_\lambda(x_1, \dots, x_n) = (\lambda^{d_1} x_1, \dots, \lambda^{d_n} x_n),
\]
where \(d_i \in \mathbb{N}\) is the \textit{homogeneity} of the variable \(x_i\) in \(\G\) (see \cite{folland_stein}, Chapter 1), and is given by
\[
d_j = i \quad \text{whenever} \, h_{i-1} + 1 \leq j \leq h_i.
\]
Hence, \(1 = d_1 = \dots = d_{m_1} < d_{m_1+1} = 2 \leq \dots \leq d_n = \kappa\).

If \(f\) is a real function defined on \(\G\), we denote by \(\ccheck f\) the function defined by \(\ccheck f(p) := f(p^{-1})\).

Following \cite{folland_stein}, we adopt the following multi-index notation for higher-order derivatives. If \(I = (i_1, \dots, i_n)\) is a multi-index, we set
\[
X^I = X_1^{i_1} \cdots X_n^{i_n}.
\]
By the Poincar\'e-Birkhoff-Witt theorem (see, e.g., \cite{bourbaki}, I.2.7), the differential operators \(X^I\) form a basis for the algebra of left-invariant differential operators on \(\G\). Moreover, we define the order of the differential operator \(X^I\) as \(|I| := i_1 + \cdots + i_n\), and its degree of homogeneity with respect to dilations as \(d(I) := d_1 i_1 + \cdots + d_n i_n\). 

%
Again, following \cite{folland_stein}, we define the group convolution in \(\G\). If \(f \in \mathcal{D}(\G)\) and \(g \in L^1_{\mathrm{loc}}(\G)\), we set
\[
f \ast g(p) := \int f(q)g(q^{-1} p) \, dq \quad \text{for} \, p \in \G.
\]
It is important to note that, if \(g\) is a smooth function and \(L\) is a left-invariant differential operator, then
\[
L(f \ast g) = f \ast Lg.
\]
The convolution is also well defined when \(f, g \in \mathcal{D}'(\G)\), provided at least one of them has compact support. In this case, the following identities hold:
\begin{equation}\label{convolutions var}
\langle f \ast g, \varphi \rangle = \langle g, \ccheck f \ast \varphi \rangle \quad \text{and} \quad \langle f \ast g, \varphi \rangle = \langle f, \varphi \ast \ccheck g \rangle
\end{equation}
for any test function \(\varphi\).

If \(f \in \mathcal{E}'(\G)\) and \(g \in \mathcal{D}'(\G)\), then for \(\psi \in \mathcal{D}(\G)\), we have
\[
\langle (X^I f) \ast g, \psi \rangle = \langle X^I f, \psi \ast \ccheck g \rangle = (-1)^{|I|} \langle f, \psi \ast (X^I \ccheck g) \rangle = (-1)^{|I|} \langle f \ast \ccheck X^I \ccheck g, \psi \rangle.
\]

Let \(1 \leq p \leq \infty\) and \(m \in \mathbb{N}\), and let \(W^{m,p}_{\mathrm{Euc}}(U)\) denote the usual Sobolev space. We also recall the definition of the (integer order) Folland-Stein Sobolev space (see, e.g., \cite{folland} and \cite{folland_stein} for a general presentation).

\begin{definition}\label{integer spaces}
If \(U \subset \G\) is an open set, \(1 \leq p \leq \infty\), and \(m \in \mathbb{N}\), the space \(W^{m,p}(U)\) consists of all \(u \in L^p(U)\) such that
\[
X^I u \in L^p(U) \quad \text{for all multi-indices \(I\) with \(d(I) \leq m\)},
\]
endowed with the norm
\[
\|u\|_{W
^{m,p}(U)} := \sum_{d(I) \leq m} \|X^I u\|_{L^p(U)}.
\]
When \(p = 2\), we will simply write \(H^m(U) = W^{m,2}(U)\).
\end{definition}

\begin{theorem}\label{folland stein varia}
Let \( U \subset \G \) be an open set, \( 1 \le p \le \infty \), and \( m \in \mathbb{N} \). Then:
\begin{itemize}
    \item[i)] \( W^{m,p}(U) \) is a Banach space.
\end{itemize}
In addition, if \( p < \infty \), the following hold:
\begin{itemize}
    \item[ii)] \( W^{m,p}(U) \cap C^\infty(U) \) is dense in \( W^{m,p}(U) \);
    \item[iii)] If \( U = \G \), then \( \mathcal{D}(\G) \) is dense in \( W^{m,p}(U) \);
    \item[iv)] If \( 1 < p < \infty \), then \( W^{m,p}(U) \) is reflexive;
    \item[v)] \( W^{m,p}_{\mathrm{Euc, loc}}(U) \subset W^{m,p}(U) \), i.e., for any \( V \subset\subset U \) and for any \( u \in W^{m,p}_{\mathrm{Euc, loc}}(U) \),
    \[
    \| u\|_{W^{m,p}(V)} \le C_V \| u \|_{W^{m,p}_{\mathrm{Euc, loc}}(U)};
    \]
    \item[vi)] \( W^{\kappa m, p}(U) \subset W^{m,p}_{\mathrm{Euc, loc}}(U) \), i.e., for any \( V \subset \subset U \) and for any \( u \in W^{\kappa m,p}(U) \),
    \[
    \| u \|_{W^{m,p}_{\mathrm{Euc}}(V)} \le C_V \| u \|_{W^{\kappa m,p}(U)}.
    \]
\end{itemize}
\end{theorem}

\begin{definition}\label{norminvar}
Let \( \G \) be a Carnot group. A \emph{homogeneous norm} \( \|\cdot\| \) on \( \G \) is a continuous function \( \|\cdot\| : \G \to [0, +\infty) \) such that:
\begin{equation}\label{norm1}
\begin{split}
    &\| p \| = 0 \quad \iff p = 0 \,; \\
    &\| p^{-1} \| = \| p \|; \\
    &\| \delta_\lambda(p) \| = \lambda \| p \|; \\
    &\| p \cdot q \| \le \| p \| + \| q \|;
\end{split}
\end{equation}
for all \( p, q \in \G \) and all \( \lambda > 0 \).

A homogeneous norm induces a homogeneous left-invariant distance \( d \) in \( \G \) in a standard way. If \( p \in \G \) and \( r > 0 \), we denote by \( B_d = B_d(p, r) \) the open \( d \)-ball centered at \( p \) with radius \( r \).
\end{definition}

In a Carnot group \( \G \), we shall consider in particular the  homogeneous norm defined in the following theorem.

\begin{theorem}[see \cite{FSSC_step2}]\label{distinfinity}
Let \( \G = V_1 \oplus \cdots \oplus V_\kappa \) be a Carnot group. Let \( \|\cdot\|_{V_1}, \dots, \|\cdot\|_{V_\kappa} \) be fixed Euclidean norms on the layers.

Then there exist constants \( \varepsilon_1, \dots, \varepsilon_\kappa \), with \( \varepsilon_1 = 1 \) and \( \varepsilon_2, \dots, \varepsilon_\kappa \in (0,1] \), depending only on the group \( \G \) and the norms \( \|\cdot\|_{V_1}, \dots, \|\cdot\|_{V_\kappa} \), such that the functions
\begin{equation}\label{distanzainfinito}
    \|x\|_\infty := \max_j \varepsilon_j \left( \| x^{(j)} \|_{V_j} \right)^{1/j}
\end{equation}
are homogeneous norms on \( \G \).

We denote by \( d_\infty \) the homogeneous left-invariant distance associated with \( \|\cdot\|_\infty \) and by \( B_\infty \) the metric balls of \( d_\infty \).

We stress that the balls \( B_\infty(e,r) \) are convex.
\end{theorem}

The vectors of \( V_1 \), also called horizontal vectors, define by left translations the \textit{horizontal bundle}, which we also denote by \( V_1 \). A section of the horizontal bundle is called a horizontal vector field.

If \( F = \sum_{i=1}^{m_1} F_i \, X_i \) is a horizontal vector field,
\[
F \in L^1_{\mathrm{loc}}(\G, V_1),
\]
we define
\[
\mathrm{div}_\G F := \sum_j X_j F_j
\]
in the sense of distributions.

\section{Main result}\label{main result and proof}
The main result of this note is stated in the following theorem.
\begin{theorem}\label{main result} 
Let $d$ be a left-invariant distance on a Carnot group associated with a homogeneous norm. Suppose $1 \leq p < Q$ and $\lambda > 1$. Set $B := B_d(e,1)$ and $B' := B_d(e,\lambda)$. If $f \in L^p(B)$ is compactly supported and satisfies
\[
\int_B f(p) \, dp = 0,
\]
then there exists a compactly supported horizontal vector field $F \in L^q(B', V_1)$, where:
\begin{itemize}
    \item[i)] $1 \leq q \leq \frac{pQ}{Q-p}$ if $p > 1$, or
    \item[ii)] $1 \leq q < \frac{Q}{Q-1}$ if $p = 1$,
\end{itemize}
such that
\[
f = \mathrm{div}_\G F \quad \text{in } B.
\]
Additionally, there exists a constant $C = C(p, q, \lambda, B)$, independent of $f$, such that
\[
\|F\|_{L^q(B', \bigwedge\nolimits_{1}{\mfrak g_1})} \leq C \|f\|_{L^p(B)}.
\]
If $p > 1$ and $q = \frac{pQ}{Q-p}$, then the constant $C$ does not depend on $B$.
\end{theorem}

Our proof of Theorem~\ref{main result} involves several steps and relies on Sobolev inequalities for differential forms in Rumin's complex. In the next subsection, we recall the key features of the Rumin's complex.

\subsection{Rumin's Complex}\label{rumin complex}

Let $\mfrak g$ be the Lie algebra of the Carnot group $\G$. The dual space of $\mfrak g$ is denoted by $\cov 1$. The basis dual to $\{X_1, \dots, X_n\}$ is the family of covectors $\{\theta_1, \dots, \theta_n\}$.

Following Federer (see \cite {federer} 1.3), the exterior algebras of 
$\mfrak g$ and of $\cov 1$ are the graded algebras indicated as
$\displaystyle\vet \ast =\bigoplus_{k=0}^{n} \vet k
$
 and 
 $\displaystyle  \cov \ast 
=\bigoplus_{k=0}^{n} \cov k
$
  where
$
 \vet 0 = \cov 0 =\R
$
and, for $1\leq k \leq n$,
\begin{equation*}
\begin{split}
         \vet k& :=\mathrm {span}\{ X_{i_1}\wedge\dots \wedge X_{i_k}: 1\leq
i_1< \dots< i_k\leq n\},   \\
         \cov k& :=\mathrm {span}\{ \gt_{i_1}\wedge\dots \wedge \gt_{i_k}:
1\leq i_1< \dots< i_k\leq n\}.
\end{split}
\end{equation*}
The elements of $\vet k$ and $\cov k$ are called  \emph{$k$-vectors} and \emph{
$k$-covectors}. 

We denote by $\Theta^k$ the basis $\{ \gt_{i_1}\wedge\dots \wedge \gt_{i_k}
:
1\leq i_1< \dots< i_k\leq n\}$
of $  \cov k$.

We denote also by $dV:= \gt_{1}\wedge\dots \wedge \gt_{n}$ the volume form
associated with our adapted basis of $\mathfrak g$. Obviously,
$ \cov n :=\mathrm {span}\{ dV\}$.

The dual space $\bigwedge^1(\vet k)$ of 
$\vet k$ can be naturally identified with $\cov k$.
The action of a $k$-covector $\gf$ on a $k$-vector $v$ is denoted as $\Scal
\gf v$.

The inner product $\scalp{\cdot}{\cdot}{} $ extends canonically to 
$\vet k $ and to $\cov k$ making the bases
$X_{i_1}\wedge\dots \wedge X_{i_k}$ and $\gt_{i_1}\wedge\dots \wedge
\gt_{i_k}$ orthonormal.

\begin{definition}\label{hodge}
We define linear isomorphisms (Hodge duality: see  \cite{federer} 1.7.8) 
\begin{equation*}
\ast : \vet k \longleftrightarrow \vet{n-k} \quad \text{and} \quad
\ast : \cov k \longleftrightarrow \cov{n-k},
\end{equation*}
for $1\leq k\leq n$, putting, for $ v = \sum_I v_I X_I $ and $\gf = 
\sum_I \gf_I \gt_I $,
\[
 \ast v := \sum\nolimits_I v_I (\ast X_I )
\quad \text {and} \quad \ast\gf := \sum\nolimits_I \gf_I (\ast \gt_I )
\]
where
\[
\ast X_I := (-1)^{\gs(I)} X_{ I^\ast}\quad \text {and} \quad
\ast \gt_I := (-1)^{\gs(I)} \gt_{ I^\ast}
\]
with $I=\{ i_1, \cdots, i_k\}$, $1\leq i_1< \cdots< i_k\leq n$, 
$X_I=X_{i_1}\wedge\cdots \wedge X_{i_k}$, 
$\gt_I=\gt_{i_1}\wedge\cdots \wedge \gt_{i_k}$, $ I^\ast=\{ 
i^\ast_1<\cdots < i^\ast_{n-k}\} = \{1, \cdots, n \}\setminus 
I$ and $\gs(I)$ is the number of couples $(i_h, i^\ast_\ell)$ with 
$i_h>  i^\ast_\ell$.
\end{definition}

Notice that, if $v=v_1\wedge \dots \wedge v_k$ is a simple $k$-vector, then $\ast v$ is a simple $(n-k)$-vector.  If $v\in \vet k$ we define $v^\natural \in \cov k$ by the identity
$
\Scal {v^\natural} w := \scal v w ,
$
and analogously we define $\gf^\natural\in \vet k$ for $\gf \in \cov k$.

\begin{definition} If $\alpha\in \cov 1$, $\alpha\neq 0$, 
 we say that $\alpha$ has \emph{pure weight $k$}, and we write
$w(\alpha)=k$, if $\alpha^\natural\in V_k$. Obviously,
$$
w(\alpha)=k\quad\mbox{if and only if}\quad \alpha = \sum_{j=h_{k-1}+1}^{h_{k}}
\alpha_j\theta_j,
$$
with $\alpha_{h_{k-1}+1},\dots, \alpha_{h_{k}}\in\mathbb R$. More generally, if
$\alpha\in \cov h$, we say that $\alpha$ has pure weight $k$ if $\alpha$ is
a linear combination of covectors $\theta_{i_1}\wedge\cdots\wedge\theta_{i_h}$
with $w(\theta_{i_1})+\cdots + w(\theta_{ i_h})=k$.
\end{definition}
\begin{remark}[see \cite{BFTT}, Remark 2.6]\label{orthogonality}
If $\alpha,\beta \in \cov h$ and $w(\alpha)\neq w(\beta)$, then
$\scal{\alpha}{\beta}=0$. 

\end{remark}
We have 
\begin{equation}\label{deco cov}
\cov h = \bigoplus_{p=M_h^{\mathrm{min}}}^{M_h^{\mathrm{max}}} \covw {h}{p},
\end{equation}
where $\covw {h}{p}$ is the linear span of the $h$--covectors of weight $p$
and $M_h^{\mathrm{min}}$, $M_h^{\mathrm{max}}$ are respectively the smallest
and the largest weight of $h$-covectors.

Since the elements of the basis $\Theta^h$ have pure weights, a basis of
$ \covw {h}{p}$ is given by $\Theta^{h,p}:=\Theta^h\cap \covw {h}{p}$
(in the Section \ref{preliminar}, we called such a basis an adapted basis).

We denote by  $\Omega^{h,p} $ the vector space of all
smooth $h$--forms in $\G$ of pure weight $p$, i.e. the space of all
smooth sections of $\covw {h}{p}$. We have
\begin{equation}\label{deco forms}
\Omega^h = \bigoplus_{p=M_h^{\mathrm{min}}}^{M_h^{\mathrm{max}}} \Omega^{h,p}.
\end{equation}

\begin{lemma}\label{d0 left} We have
$d(\covw{h}{p})= \covw{h+1}{p}$, i.e., if $\alpha\in \covw{h}{p}$ is a left invariant 
$h$-form of weight p,
then $w(d\alpha)=w(\alpha)$.
\end{lemma}
\begin{proof} See \cite{rumin_grenoble}, Section 2.1.
\end{proof}

Let now $\alpha\in \Omega^{h,p}$ be a (say) smooth form
of pure weight $p$. We can write
$$
\alpha= \sum_{\theta^h_i\in\Theta^{h,p}}\alpha_{i}\, \theta_i^h,\quad
\mbox{with } \alpha_i\in \mc E (\G).
$$
Then
$$
d\alpha=\sum_{\theta^h_i\in\Theta^{h,p}}
\sum_j(X_j\alpha_i)\theta_j\wedge\theta^h_i +
\sum_{\theta^h_i\in\Theta^{h,p}}\alpha_id\theta^h_i.
$$
Hence we can write
$$
d=d_0+d_1+\dots+d_\kappa,
$$
where $$d_0\alpha = \sum_{\theta^h_i\in\Theta^{h,p}}\alpha_id\theta_i^h$$ does not increase the weight,
$$d_1\alpha= \sum_{\theta^h_i\in\Theta^{h,p}}\sum_{j=1}^{m_1}(X_j\alpha_i)\theta_j\wedge\theta_i^h$$
increases the weight of $1$, and, more generally,
$$d_k\alpha= \sum_{\theta^h_i\in\Theta^{h,p}}\sum_{w(\theta_j)=k}(X_j\alpha_i)\theta_j\wedge\theta_i^h
\quad k=1,\dots,\kappa.$$

In particular, $d_0$ is an algebraic operator.

\begin{definition}\label{E0}
If $0\le h\le n$ and we denote by $d_0^*$ the algebraic adjoint of $d_0$, we set
$$
E_0^h:=\ker d_0\cap\ker d_0^* = \ker d_0\cap (\mathrm{Im}\; d_0)^{\perp}
\subset \Omega^h.
$$

Since the construction of $E_0^h$ is left invariant, this space of forms
can be viewed as the space of sections of a fiber bundle, generated by
left translation and still denoted by $E_0^h$.
\end{definition}

We denote by $N_h^{\mathrm{min}}$ and $N_h^{\mathrm{max}}$ the
minimum and the maximum, respectively, of the weights of forms in $E_0^h$.

If we set $E_0^{h,p}:= E_0^h\cap \Omega^{h,p}$, then
\begin{equation*}
E_0^h = \bigoplus_{p=N_h^{\mathrm{min}}}^{N_h^{\mathrm{max}}} E_0^{h,p}.
\end{equation*}

We notice that also the space of forms $E_0^{h,p}$
 can be viewed as the space of smooth sections of a
suitable fiber bundle generated by left translations,
that we still denote by $E_0^{h,p}$.

As customary, if $\Omega\subset\G$ is an
open set, we denote by $\mc E(\Omega, E_0^h)$ the space
of smooth sections of $E_0^h$. 

The spaces $
\mc D(\Omega, E_0^h)$ and
$
\mc S(\G, E_0^h)$
are defined analogously.

Since both $E_0^{h,p}$ and $E_0^h$ are left invariant as
$\cov h$, they are subbundles of $\cov h$ and inherit the
scalar product on the fibers. 

In particular, we can obtain a left invariant orthonormal basis 
$\Xi_0^h=\{\xi_j^h\}$ of $E_0^h$ such that 
\begin{equation}\label{e0 cov basix}
\Xi^h_0 = \bigcup_{p=N_h^{\mathrm{min}}}^{N_h^{\mathrm{max}}} \Xi^{h,p}_0,
\end{equation}
where $ \Xi^{h,p}_0:= \Xi^h\cap\covw{h}{p}$ is a left invariant orthonormal basis of
$E_0^{h,p}$. All the elements of $\Xi^{h,p}_0$ have pure
weight $p$.

Once the basis $\Theta_0^h$ is chosen, the spaces 
$\mc E(\Omega, E_0^h)$,
$\mc D(\Omega, E_0^h)$,
$\mc S(\G, E_0^h)$ can be identified with
$\mc E(\Omega)^{\dim E_0^h}$,
$\mc D(\Omega)^{\dim E_0^h}$,
$S(\G)^{\dim E_0^h}$, respectively.

\begin{proposition} [\cite{rumin_grenoble}]\label{Ehodge} If $0\le h\le n$ and $*$ denote the Hodge
duality (see Definition \ref{hodge}), then
$$
*E_0^h = E_0^{n-h}.
$$
\end{proposition}

By a simple linear algebra argument we can prove the following 
lemma.

\begin{lemma}\label{d_0}
If $\beta\in\Omega^{h+1}$, then there exists a unique $\alpha\in
\Omega^h\cap (\ker d_0)^\perp$ such that
$$
d_0^*d_0\alpha = d_0^*\beta.\quad\mbox{We set}\quad
\alpha :=d_0^{-1}\beta.
$$
\end{lemma}

\begin{remark}\label{d0 varie}
We stress that $d_0^{-1}$ is an algebraic operator, like $d_0$ and $\delta_0$,
\end{remark}

\begin{lemma}[\cite{rumin_grenoble}]\label{rumin 2.5}
The map $d_0^{-1}d$ induces an isomorphism from $\mc R (d_0^{-1})$
to itself. In addition, there exist a differential operator 
$$P=\sum_{k=1}^N(-1)^kD^k,\quad\mbox{$N\in\mathbb N$ suitable,}$$ such that
$$
P d_0^{-1}d = d_0^{-1}d P = \mathrm{Id}_{\mc R( d_0^{-1})}.
$$
We set $Q:=Pd_0^{-1}$.
\end{lemma}

\begin{remark}\label{P and weights}
If  $\alpha$ has pure weight $k$, then $P\alpha$
is a sum of forms of pure weight  greater or equal to $k$.
\end{remark}

We state now the following key results. 
\begin{theorem}[\cite{rumin_grenoble}]\label{main rumin bis}
The de Rham complex $(\Omega^*,d)$ 
splits as the direct sum of two sub-complexes $(E^*,d)$ and
$(F^*,d)$, with
$$
E:=\ker d_0^{-1}\cap\ker (d_0^{-1}d)\quad\mbox{and}\quad
F:= \mc R(d_0^{-1})+\mc R (dd_0^{-1}),
$$
such that
\begin{itemize}
\item[i)] The projection $\Pi_E$ on $E$ along $F$ is given by $\Pi_E=Id-Qd-dQ$.
In particular, $\Pi_E$ is a differential operator of order $s\ge 0$ in the horizontal derivatives,
where $s$ depends on $\G$ and on the degree of the forms  it acts on.
\item[ii)] If $\;\Pi_{E_0}$ is the orthogonal projection from $\Omega^*$
on $E_0^*$, then $\Pi_{E_0}\Pi_{E}\Pi_{E_0}=\Pi_{E_0}$ and
$\Pi_{E}\Pi_{E_0}\Pi_{E}=\Pi_{E}$.
\end{itemize}
 \end{theorem}

\begin{theorem}[\cite{rumin_grenoble}]\label{main rumin}
If we set 
$$d_c:=\Pi_{E_0}\, d\,\Pi_{E},$$
 then $d_c: E_0^h\to E_0^{h+1}$ satisfies
\begin{itemize}
\item[i)] $d_c^2=0$;
\item[ii)] the complex $E_0:=(E_0^*,d_c)$ is exact;
\end{itemize}
\end{theorem}

In particular, if $h=0$ and $f\in E_0^0=\mc E(\G)$, then $d_cf = \sum_{i=1}^m (X_i f) \theta_i^1$ is the
{\sl horizontal differential of $f$}.

In addition, by Proposition \ref{Ehodge}, $E_0^n = \{f\, dV,\; f\in\mc E(\G)\}$.

 \begin{remark}[see \cite{BFTT} Remark 2.17]\label{cor main rumin} We have
 \begin{equation}\label{cor main rumin 1}
 d\,\Pi_E = \Pi_E d.
 \end{equation}
 \end{remark}

It follows from Proposition 2.18 of \cite{BFTT} that, if $\alpha\in E_0^h$ has weight
$p$, then
$$\Pi_E\alpha = \alpha +
\mbox{terms of weight greater than $p$.}$$
\begin{remark}\label{PiE on top degree} In particular, if $\alpha\in E_0^n$ (and therefore has weight $Q$), then
$\Pi_E\alpha = \alpha$, since there are no forms of weight $> Q$. 
\end{remark}
\begin{definition}\label{d*}
We denote by $d_c^*$ the $L^2$-(formal) adjoint of $d_c$.
We recall that on $E_0^h$
$$
d_c^* = (-1)^{n(h+1)+1} \ast d_c \ast.
$$
\end{definition}

\subsection{Equivalent formulation and proof of Theorem \ref{main result}}

Let us start by noticing that $d_c^*$ on 1-forms can be identified
with the horizontal divergence. Indeed,
If $F=\sum_{i=1}^{m_1}F_i \,X_i \in L^1_{\mathrm{loc}}(\G, V_1)$, we denote by
$
F^\sharp
$
the differential 1-form defined by
$$
\Scal{F^\natural}{V} = \scal{F}{V} = \sum_j\int_\G F_j V_j\, dp
$$
for any $V=\sum_{i=1}^{m_1}V_i\, X_i \in \mc D(\G, V_1)$,
i.e.
$$
F^\natural = \sum_i F_i \, \theta_i.
$$
 
 If now $\phi\in \mc D(\G)$, then (keeping in mind that
$d_c^* F^\natural$ is a 0-form)
\begin{equation}\label{july 25 eq:1bis}\begin{split}
\int_\G & \phi \, d_c^* F^\natural \, dp = \int_\G \scal{F^\natural}{ d_c\phi}\, dp
\\&
= \sum_j \int_\G F_j\, X_j \phi = - \int_\G \phi \,\mathrm{div}_\G F,
\end{split}\end{equation}
where the above identities are meant in the sense of distributions.
Hence $f= \mathrm{div}_\G F$ if and only if $d_c^* F^\natural= f $, i.e.
$$
-\ast d_c \ast  F^\natural= f .
$$
Applying Hodge operator to identity, and keeping in mind that $d_c \ast  F^\natural$ is a $n-$form and hence $\ast\ast=Id$,   we obtain
$$
-\ast \ast d_c \ast  F^\natural=  \ast f ,
$$
i.e.
$$
d_c (-\ast  F^\natural)=   f\, dV.
$$
If we set $\phi:=- \ast  F^\natural$ and $\omega:=f\, dV$, an equivalent formulation of Theorem \ref{main result} becomes:

\begin{theorem} \label{main result bis} Let $d$ be a left invariant distance on a Carnot
group associated with a homogeneous norm. Let $1\le p<Q$ and $\lambda>1$,
and set $B:=B_d(e,1)$ and $B':= B_d(e,\lambda)$. If
$\omega \in L^p(B,E_0^n)$ is compactly supported and satisfies
$$
\int_{B} \omega = 0,
$$
then there exists a compactly supported differential form 
$\phi\in L^{q}(B', E_0^{n-1})$ with
\begin{itemize}
\item[i)] $1\le q \le pQ/(Q-p)$ if $p>1$ 
\end{itemize}
 or
\begin{itemize}
\item[ii)] $1\le q < Q/(Q-1)$ if $p=1$,
\end{itemize}
so that 
$$
d_c\phi = \omega\qquad \mbox{in $B$}.
$$

In addition, there exists $C=C(p,q,\lambda,B)$ independent of $\omega$ such that
$$
\| \phi \|_{L^{q}(B', E_0^{n-1})}
\le C \|\omega\|_{L^p(B,E_0^n)}.
$$
If $p>1$ and $q = pQ/(Q-p)$, the constant does not depend on $B$.

\end{theorem}

\medskip

Since different homogeneous norms are equivalent
(\cite{BLU}, Proposition 5.1.4),
without loss of generality from now on we may assume that $d=d_\infty$
and 
 for sake of simplicity, we shall write $B(p,r)$ for $B_\infty(p,r)$.
 
 \medskip

The first step in order to prove Theorem \ref{main result bis} will be to define an operator acting on
$n$-forms which inverts Rumin's differential $d_c$ (albeit with a loss of regulaity)
.
Inspired by the work of \cite{IL},  Mitrea, Mitrea and Monniaux,  in \cite{mitrea_mitrea_monniaux}, define a compact homotopy operator $J_{\mathrm{Euc},h}$ 
in Lipschitz star-shaped  domains in Euclidean space $\rn {n}$, providing an explicit representation formula
for $J_{\mathrm{Euc},h}$, together with continuity properties among Sobolev spaces. Since in this Note we are
interested on forms of top degree $n$, we recall what Theorem 4.1 of \cite{mitrea_mitrea_monniaux}
states only in this particular case.  Theorem 4.1 of \cite{mitrea_mitrea_monniaux} says that if $D\subset \rn {N}$ is 
a star-shaped Lipschitz domain, then  there exists
$$
J_{\mathrm{Euc},h} : L^{p}(D, {\bigwedge}\vphantom{!}^n) \to W^{1,p}_{\mathrm{Euc}}(D, {\bigwedge}\vphantom{!}^{n-1}) 
\hookrightarrow W^{\kappa,p}(D, E_0^{n-1}) 
$$
such that
\begin{equation}\label{july 27 eq:4}
\omega = dJ_{\mathrm{Euc},n}\omega + \big(\int_D \omega\big)\theta \, dV
 \qquad \mbox{for all $\omega\in 
\mc D(D, {\bigwedge}\vphantom{!}^n)$,}
\end{equation}
where $\theta\in\mc D(\G)$ satisfies
$$
\int_\G \theta\, dp=1.
$$
Furthermore, $J_{\mathrm{Euc},n}$ maps smooth compactly supported forms to smooth compactly supported forms.

For the sake of simplicity, from now on we drop the index $n$ - the degree of the form -
writing, e.g., $J_{\mathrm{Euc}}$ instead of $J_{\mathrm{Euc},n}$.

\medskip

To our aim, take now $D=B$. If $\omega\in \mc D(B,E_0^n)$,  with
vanishing average, we set
\begin{eqnarray}\label{may 31 eq:2}
J=\Pi_{E_0}\circ \Pi_E \circ J_{\mathrm{Euc}}   \circ \Pi_E.
\end{eqnarray}

Since $\Pi_E\omega= \omega$ on $E_0^n$, on $E_0^n$ we can also write
\begin{eqnarray}\label{may 31 eq:2 bis}
J\omega=\Pi_{E_0}\circ \Pi_E \circ J_{\mathrm{Euc}}\omega.
\end{eqnarray}
Then $J$ inverts Rumin's differential $d_c$ on forms of degree $n$ in the sense of the following result.

\begin{lemma}\label{homotopy 1} If $\alpha\in E_0^n$ is a compactly supported smooth form  in 
a ball $\tilde B$
with
$$
\int_{\tilde B} \alpha = 0,
$$
 then
\begin{equation}\label{homotopy closed}
\alpha = d_cJ\alpha.
\end{equation}
In addition, $J\alpha$ is  compactly supported in $\tilde B$.
\end{lemma} 

\begin{proof} 

By \eqref{july 27 eq:4},
\begin{equation}\label{may 4 eq:2}
\alpha = dJ_{\mathrm{Euc}}  \alpha.
\end{equation}

We recall now that $\Pi_E\Pi_{E_0} \Pi_E=\Pi_E$ and $\Pi_{E_0}\Pi_E \Pi_{E_0}=\Pi_{E_0}$.
In addition, on forms of degree $n-1$, $d\Pi_E=\Pi_Ed$.
Thus, by \eqref{may 4 eq:2},
 \begin{eqnarray*}
d_cJ\alpha&=&\Pi_{E_0}d \Pi_E\Pi_{E_0} \Pi_E  J_{\mathrm{Euc}}   \alpha=\Pi_{E_0}d \Pi_E J_{\mathrm{Euc}}  \alpha  
\\
&=&\Pi_{E_0} \Pi_E d J_{\mathrm{Euc}}   \alpha=\Pi_{E_0} \Pi_E\alpha ={ \Pi_{E_0}  \alpha}
=\alpha,
\end{eqnarray*}
since $\alpha\in E_0^n$.
Finally, if $\mathrm{supp}\,  \alpha
\subset \tilde B$, then $\mathrm{supp}\, J \alpha
\subset \tilde B$ since both $\Pi_E$ and $\Pi_{E_0}$ preserve the support.
\end{proof}

Unfortunately, the operator $J$ contains the differential operator $\Pi_E$
that yields a loss of regularity. We can get rid of this inconvenient combining
$J$ with a smoothing operator coming from an approximated homotopy
formula. The approximated homotopy
formula is based on a global homotopy identity relying on
 the inverse of Rumin's Laplacian.

Indeed, if $\omega=fdV\in \mc D(\G, E_0^n)$, we can define its sub-Laplacian as
$$
 \Delta_{\G,n} \omega:= d_cd_c^*  \omega.
 $$
Since $\ast\ast = Id $ on $n$-forms,
$$
 \Delta_{\G,n} \omega = \ast  \Delta_{\G,0} \ast \omega,
$$
and the fundamental solution  $\Delta_{\G,n}^{-1}$ of $ \Delta_{\G,n}$ is given by
$$
\Delta_{\G,n}^{-1} = \ast \Delta_{\G,0}^{-1}\ast
$$
that is associated with a kernel of type 2 (see \cite{folland}).

We are now able to prove the equivalent formulation of Theorem \ref{main result} arguing as in Theorem 5.12 of \cite{BFP2}.

\begin{proof}[Proof of Theorem \ref{main result bis}]

Suppose first that $\omega\in \mc D(B, E_0^n)$.  If $\omega$ is continued
 by zero on all of $\G$, we notice preliminarily
that
\begin{equation*}
\omega =  
 \Delta_{\G,n}  \Delta^{-1}_{\G,n}\omega = d_c (d_c^*  \Delta^{-1}_{\G,n})\omega,
\end{equation*}
where $d_c^* \Delta^{-1}_{\G,n}$ is
 associated with a matrix-valued kernel $k_1$ of type $1$
 acting on $f$.
 Keeping in mind that, by Hodge duality, $\omega$ can be identified
 with the function $f$, without loss of generality, we can treat $k_1$ as it were a scalar kernel.
 We consider a cut-off function $\psi_R$ supported in a $R$-neighborhood
of the origin, such that $\psi_R\equiv 1$ near the origin. We can write  
\begin{equation}\label{morti}
k_1=\psi_R  k_1+ (1-\psi_R)k_1.
\end{equation}
Since the kernel of $\Delta^{-1}_{\G,n}$ is of type 2,
the kernel $\psi_R  k_1$ belongs to $L^1(\G)$.
 Let us denote by $K_{1,R}$ the convolution operator associated with
$\psi_R k_1$ and by $S$ is the convolution operator associated with the kernel
\begin{equation}\label{santi}
K_S:=  d_c( (1-\psi_R)k_1 ) .
\end{equation}
It follows from \eqref{morti} that
\begin{equation}\label{sept 9 eq:1}
\omega = d_c K_{1,R} \omega + S\omega
\qquad\mbox{if $\omega\in \mc D(B,E_0^n)$.}
\end{equation}

The kernel $K_S$ is smooth. We stress also that $\mathrm{supp}\, K_{1,R} \omega$ is contained in
a $R$-neighborhood of $B$ so that
\begin{equation}\label{july 26 eq:3}
\mathrm{supp}\, K_{1,R} \omega \subset B'
\end{equation}
provided $R=R(\lambda)<d(B, \partial B')$. By \eqref{sept 9 eq:1}, also 
\begin{equation}\label{july 30 eq:1}
\mathrm{supp}\, S \omega \subset B'.
\end{equation}
Finally, by \eqref{sept 9 eq:1}, $S\omega \in E_0^n$.

The homotopy formula \eqref{sept 9 eq:1} still holds in the
sense of distributions when $\omega\in L^p$. To prove that,
 we need the following lemma.

\begin{lemma}\label{bologna 1 ottobre} Being $S$ and $K_{1,R}$ defined as above, we have:
\begin{itemize}
\item[i)] $S$ is regularizing from $\mc E'(\G)$ to
$\mc E(\G)$.
In addition, if $p,q\ge 1$ and $m\in \mathbb N\cup\{0\}$
then $S$ can be continued as a bounded map from $L^p(B,E_0^n)\cap\mc E'(B,E_0^n)$ to $W^{m,q}(B',E_0^n)$
$$
S:L^p(B,E_0^n) \to W^{m,q}(B',E_0^n).
$$
In particular, by Thorem \ref{folland stein varia}, vi),
due to the arbitrariness of the choice of $m$, we also have 
$$
S:L^p(B,E_0^n) \to W^{m,q}_{\mathrm{Euc}}(B',E_0^n);
$$

\item[ii)] if $p\ge 1$, the map $K_{1,R}$ can be continued  as a bounded map  
from $L^p(B,E_0^n)\cap\mc E'(B,E_0^n)$ to $L^p(B',E_0^n)$;
\item[iii)] if $p>1$ then the map $K_{1,R}$ can be continued  as a bounded map  
from $L^p(B,E_0^n)\cap\mc E'(B,E_0^n)$ to $W^{1,p}(B',E_0^n)$
and the identity \eqref{morti} still holds for $\omega\in L^p(B,E_0^n)\cap
\mc E'(B,E_0^n)$;
\item[iv)] the identity \eqref{morti} still holds for $\omega\in L^1(B,E_0^n)\cap
\mc E'(B,E_0^n)$ in the sense of distributions;
\item[v)] if $p>1$, then $K_{1,R}: L^p(B,E_0^n)\cap
\mc E'(B,E_0^n) \to L^{q}(B',E_0^{n-1})$ for $p\le q\le Q/(Q-1)$;
\item[vi)] $K_{1,R}: L^1(B,E_0^n)\cap
\mc E'(B,E_0^n) \to L^{q}(B',E_0^{n-1})$ for $1\le q<Q/(Q-1)$.
\end{itemize}
\end{lemma}

\begin{proof} Let us prove i). Since the kernel $K_S$ is smooth
and the convolution maps $\mc E'(\G)\times \mc E(\G)$
into $\mc E(\G)$, the operator $S$ is regularizing from $\mc E'(\G)$ to
$\mc E(\G)$ (see \cite{laurent_schwartz}, p.167). In addition,
since $B$ is bounded, then without  loss of generality we may assume
that $p=1$.

Remember $\omega = fdV$; hence we can identify $\omega$ and 
the scalar function $f$. 
We have:
\begin{equation*}\begin{split} 
\| S\omega & \|_{W^{m,q}(B',E_0^n)} = \| \omega\ast K_S\|_{W^{m,q}(B')}
\\&
= \sum_{d(I) \le m} \|\omega \ast X^I  K_S \|_{L^q(B')}  
\\&
=  \sum_{d(I) \le m}
\Big(\int_{B'} \big( \int_B |\omega(y)|\,| X^I K_S(y^{-1}x)|  \, dy\big)^q \, dx\Big)^{1/q}.
\end{split}\end{equation*}
Notice now that, if $x\in B'$ and $y\in B$, then $y^{-1}x\in B(e,1+\lambda)$. Thus,
if $\chi\in\mc D(\G)$ is a cut-off function, $\chi\equiv 1$ on $B(e,1+\lambda)$, then
$\chi  X K_S\in L^q(\G)$, so that, by Young's inequality (see Theorem \ref{folland cont}, i)), 
Proposition 1.18)
\begin{equation*}\begin{split}
\| S\omega & \|_{W^{m,q}(B',E_0^n)} \le
 \sum_{d(I) \le m}
  \| |\chi \omega(y)|\ast  |X^I K_S|  \|_{L^q(\G)}
\\&
\le C \| \omega(y) \|_{L^1(\G)} = C \| \omega(y) \|_{L^1(B)}.
\end{split}\end{equation*}

Proof of ii). By a similar argument
\begin{equation*}\begin{split}
\| K_{1,R}\omega & \|_{L^p(B',E_0^n)} \le
\| \omega\ast \psi_Rk_1  \|_{L^p(B',E_0^n)} 
\\&
\le \Big(\int_{B'} \big( \int_B  |\omega(y)|\, |\psi_Rk_1(y^{-1}x) | \, dy\big)^p \, dx\Big)^{1/p}
\\&
\le C  \|\psi_Rk_1 \|_{L^1(B(1+\lambda))}  \|\omega\|_{L^p(B,E_0^n)}.
\end{split}\end{equation*}

Proof of iii). Let $X$ be a horizontal derivative. Then,
 we have only to estimate
the $L^p$-norm of 
$$
X(\omega\ast\psi_R k_1) = \omega\ast (X\psi_R) k_1 + \omega\ast(\psi_R Xk_1).
$$
By Lemma \ref{truncation}
\begin{equation*}\begin{split}
\| \omega\ast &(X\psi_R) k_1\|_{L^p(B')}\le C \| \omega\ast (X\psi_R) k_1\|_{L^{pQ/(Q-p)}(B')}
\\&
\le C \| \omega\ast (X\psi_R) k_1\|_{L^{pQ/(Q-p)}(\G )}
\le C \| \omega\|_{L^{p}(\G )} = C \| \omega\|_{L^{p}(B )};
\end{split}\end{equation*}
analogously, since $Xk_1$ is a kernel of type 0,
$$
\| \omega\ast(\psi_R Xk_1)\|_{L^p(B')}\le 
 C \| \omega\|_{L^{p}(\G )}
\le C \| \omega\|_{L^{p}(B )}.
$$
Finally, since $\omega$ is compactly supported
in $B$, it can be approximated in $L^p(B)$ by a sequence $(\omega_k)_{k\in\mathbb N}$ in 
$\mc D(B)$. Thus 
$$
d_cK_{1,R}\omega_k \to d_cK_{1,R}\omega\qquad\mbox{in $L^p(\G)$ as $k\to\infty$.}
$$
In addition, by  i),
$$
S\omega_k \to S\omega \qquad\mbox{in $L^p(\G)$ as $k\to\infty$,}
$$
and iii) is proved

Proof of iv).  Take a sequence $(\omega_k)_{k\in\mathbb N}$ as in the proof of iii). By ii) 
$$
K_{1,R}\omega_k \to K_{1,R}\omega\qquad\mbox{in $L^p(\G)$ as $k\to\infty$.}
$$
In particular, $d_cK_{1,R}\omega_k \to d_cK_{1,R}\omega$ in the sense of distributions.
Then iv) follows from \eqref{morti}.

Proof of v). The statement follows by Lemma \ref{truncation}.

Proof of vi). The statement follows by Remark \ref{truncation 2}

\end{proof}

Let us resume the proof of Theorem \ref{main result}.
Since $S$ is a smoothing operator, then $S\omega\in 
\mc D(B',E_0^n)$, keeping also in mind 
that $S\omega$ is supported in $B' $ (see \eqref{july 30 eq:1}).

 We notice also that for any $p\ge 1$, 
 $S\omega$ 
 has vanishing average, since $\omega $ has vanishing average.
Indeed, take $\chi\in D(\G)$, $\chi\equiv 1$ on $B'$.  Again identify  $\omega=fdV$ with 
the scalar function $f$, we have, by Lemma \ref{bologna 1 ottobre} iii) and iv),
that the homotopy formula \eqref{sept 9 eq:1} holds in the sense of distributions.
Therefore
\begin{equation*}\begin{split}
\int_{B'} & S\omega \, dV = \int_{B'} \chi S\omega\, dV 
\\&
= \int_{\G}  \chi \omega\, dV + \int_\G (d_c\chi) \wedge K_{1,R}\omega
=0,
\end{split}\end{equation*}
since $d_c\chi=0$ on $\supp K_{1,R}\omega$.

Since $S\omega$
has vanishing average,  we can apply \eqref{homotopy closed} to $\alpha:=S\omega$
and we get $S\omega = d_cJS\omega$, where $J$ is defined in \eqref{may 31 eq:2}. By Lemma \ref{homotopy 1},
 $JS\omega$ is supported in $B'$. Thus, if we set
$\phi:= (JS+K_{1,R})\omega$, then $\phi$ is supported in $B'$. Moreover $d_c\phi = d_cJS\omega + d_cK_{1,R}\omega = S\omega +
\omega - S\omega = \omega$. 

Remember now that, by Theorem \ref{main rumin bis}-i), $\Pi_E$ on forms of degree $(n-1)$
 is a differential operator of order $s\ge 0$ in the horizontal
derivatives. Thus, by Lemma \ref{bologna 1 ottobre},
\begin{equation}\label{pq estimates}\begin{split}
\|\phi & \|_{L^q(B',E_0^{n-1})}  \le \|JS\omega\|_{L^q(B',E_0^{n-1})}+ \| K_{1,R}\omega\|_{L^q(B',E_0^{n-1})} 
\\&
\le  \|JS\omega\|_{L^q(B',E_0^{n-1})}+C \|   \omega\|_{L^p(B',E_0^{n})}
\\&
\le  \|S\omega\|_{W^{s-1,q}_{\mathrm{Euc}}(B',E_0^{n-1})}+C \|   \omega\|_{L^p(B',E_0^{n})}
\\&
\le  \|S\omega\|_{W^{s,q}_{\mathrm{Euc}}(B',E_0^{n-1})}+C \|   \omega\|_{L^p(B',E_0^{n})}
\\&
\le C\big(  \|S\omega\|_{W^{\kappa s,q}(B,E_0^{h})}+ \|   \omega\|_{L^p(B',E_0^{n})}
\big)
\\&
\le C \|   \omega\|_{L^p(B',E_0^{n})}.
\end{split}\end{equation}

This completes the proof of the theorem.

\end{proof}

\section{Appendix A: kernels in Carnot groups}\label{carnot kernels}

Following \cite{folland} and \cite{folland_stein}, we now recall the notion of a 
\textit{kernel of type $\mu$} and some related properties, 
as outlined in Propositions \ref{kernel} and \ref{folland cont} below. 
For these results, we refer to Section 3.2 of \cite{BFP3}.
\begin{definition}\label{type} A kernel of type $\mu$ is a 
homogeneous distribution of degree $\mu-Q$
(with respect to group dilations),
that is smooth outside of the origin.

The convolution operator with a kernel of type $\mu$
is still called an operator of type $\mu$.
\end{definition}

\begin{proposition}\label{kernel}
Let $K\in\mc D'(\G)$ be a kernel of type $\mu$.
\begin{itemize}
\item[i)] $\ccheck K$ is again a kernel of type $\mu$;
\item[ii)] $WK$ and $KW $ are associated with  kernels of type $\mu-1$ for
any horizontal derivative $W$;
\item[iii)]  If $\mu>0$, then $K\in L^1_{\mathrm{loc}}(\G)$.
\end{itemize}
\end{proposition}

\begin{theorem}  \label{folland cont} We have:
\begin{itemize}
         \item[i)] Hausdorff-Young inequality holds, i.e.,  if $f\in L^p(\G)$, $g\in L^q(\G)$, $1\le p, q,r \le \infty$ and $\frac1p + \frac1q - 1 = \frac1r$, then $f\ast g\in L^r(\G)$
         (see \cite{folland_stein}, Proposition 1.18) .
	\item[ii)] If $K$ is a kernel of type $0$, $1<p<\infty$,  then the mapping $\, T:u\to u\ast K$ defined for $u\in\mc D(\G)$ extends to a bounded operator on $L^p(\G)$
	(see \cite{folland}, Theorem 4.9).
	\item[iii)] Suppose $0<\mu <Q$, $1<p<Q/\mu$ and $\frac{1}{q}=\frac{1}{p}-\frac{\mu}{Q}$. Let $K$ be a kernel of type $\mu$. If $u\in L^p(\G)$ the convolutions $u\ast K$ and $K\ast u$ exists a.e. and are in $L^q(\G)$ and there is a constant $C_p>0$ such that 
	$$
	\|u\ast K\|_q\le C_p\|u\|_p\quad \mathrm{and}\quad \| K\ast u\|_q\le C_p\|u\|_p\,
	$$
	(see \cite{folland}, Proposition 1.11).

\end{itemize} 
\end{theorem}

\begin{lemma}(see Lemma 3.5 in \cite{BFP2}) \label{truncation} Suppose $0< \mu<Q$.
If $K$ is a kernel of type $\mu$
and $\psi \in \mc D(\G)$, $\psi\equiv 1$ in a neighborhood of the origin, then
the statement iii) of Theorem \ref{folland cont} still
holds if we replace $K$ by $\psi K$ or $(1-\psi )K$. 

Analogously, if $K$ is a kernel of type 0 and $\psi \in \mc D(\G)$,
then statement ii) of the same theorem still
holds if we replace $K$ by $\psi K$ or $(\psi-1) K$.
\end{lemma}

\begin{definition} Let $f$ be a measurable function on $\G$. If $t>0$ we set
$$
\lambda_f(t) = |\{|f|>t\}|.
$$
If $1\le r\le\infty$ and
$$
 \sup_{t>0}t^r \lambda_f(t)  <\infty,
$$
we say that $f\in L^{r,\infty}(\G)$.
\end{definition}

\begin{definition}\label{M}
Following \cite{BBD}, Definition A.1, if $1<r<\infty$, we set 
$$
\| u\|_{M^r} : = \inf \{C\ge 0 \, ; \, \int_K |u| \, dx \le C |K|^{1/r'}\;
\mbox{for all $L$-measurable set $K\subset \G$}\}.
$$
and $M^r = M^r(\G)$ is the set of measurable functions $u$ on $\G$ satisfying $\| u\|_{M^r}<\infty$.
\end{definition}

Repeating verbatim the arguments of \cite{BBD}, Lemma A.2, we obtain

\begin{lemma}\label{marc alternative} If $1<r<\infty$, then
$$
\dfrac{(r-1)^r}{r^{r+1}}  \| u\|_{M^r }^r  \le \sup_{t >0} \{t^r | \{|u|>t\} |\, \} \le  \| u\|_{M^r }^r.
$$
In particular, if $1<r<\infty$, then  $M^r  = L^{r,\infty}(\G)$.
\end{lemma}

\begin{corollary}\label{marc alternative coroll}  If $1\le q <r$, then $M^r \subset L^q_{\mathrm{loc}} (\G)\subset L^1_{\mathrm{loc}} (\G)$.

\end{corollary}

\begin{proof} By Lemma \ref{marc alternative}, if $u\in M^r$ then $|u|^q\in M^{r/q}$, and we can conclude
thanks to Definition \ref{M}.

\end{proof}

\begin{lemma}\label{convolutions 2} Let $K$ be a kernel of type $\mu\in (0,Q)$. Then for all $f\in L^1(\G)$ we have $f\ast K\in M^{Q/(Q-\mu)} $
and there exists $C>0$ such that 
$$
 \| f\ast K\|_{M^{Q/(Q-\mu)}} \le C\|f \|_{L^1({\G})}
   $$
for all $f\in L^1(\G)$. In particular, by Corollary \ref{marc alternative coroll}, 
if $1\le q < Q/(Q-\mu)$, then $f\ast K\in
L^{q}_{\mathrm{loc}}(\G) \subset  L^1_{\mathrm{loc}}(\G)$.
\end{lemma}

As in \cite{BFP3}, Remark 3.10, we have:

\begin{remark}\label{truncation 2} Suppose $0< \mu<Q$.
If $K$ is a kernel of type $\mu$
and $\psi \in \mc D(\G)$, $\psi\equiv 1$ in a neighborhood of the origin, then
the statements of Lemma \ref{convolutions 2} still
hold if we replace $K$ by $(1-\psi )K$ or by $\psi K$. 

\end{remark}

\section*{Acknowledgements}

A.B. and B.F. are supported by the University of Bologna, funds for selected research topics.
A.B. is supported by PRIN 2022 F4F2LH - CUP J53D23003760006 ``Regularity problems in sub-Riemannian structures'', and by GNAMPA of INdAM (Istituto Nazionale di Alta Matematica ``F. Severi''), Italy

P.P. is supported by Agence Nationale de la Recherche, ANR-22-CE40-0004 GOFR.

\bibliographystyle{amsplain}

\bibliography{BFP7_final}

\bigskip
\tiny{
\noindent
Annalisa Baldi and Bruno Franchi 
\par\noindent
Universit\`a di Bologna, Dipartimento
di Matematica\par\noindent Piazza di
Porta S.~Donato 5, 40126 Bologna, Italy.
\par\noindent
e-mail:
annalisa.baldi2@unibo.it, 
bruno.franchi@unibo.it.
}

\medskip

\tiny{
\noindent
Pierre Pansu 
\par\noindent  Universit\'e Paris-Saclay, CNRS, Laboratoire de math\'ematiques d'Orsay
\par\noindent  91405, Orsay, France.
\par\noindent 
e-mail: pierre.pansu@universite-paris-saclay.fr
}

\end{document}